\documentclass{hha}

\usepackage[alphabetic]{amsrefs}
\usepackage[all]{xy}
\newdir{ >}{{}*!/-8pt/@{>}} 
\SelectTips{cm}{} 
\usepackage{xcolor}
\colorlet{darkgreen}{green!40!black}
\usepackage{ifpdf}
\ifpdf \usepackage[pdftex]{hyperref}
\else \usepackage[ps2pdf]{hyperref} 
      \usepackage{breakurl}
\fi
\hypersetup{colorlinks=true,urlcolor=blue,citecolor=darkgreen,
linkcolor=darkgreen,linktocpage,bookmarksnumbered,unicode}

\newcommand{\cA}{\mathcal{A}}
\newcommand{\F}{\mathbb{F}}
\newcommand{\lra}{\longrightarrow}
\newcommand{\lla}{\longleftarrow}
\newcommand{\ds}{\displaystyle}

\DeclareMathOperator{\fib}{fib}
\DeclareMathOperator{\Ext}{Ext}

\newtheorem{question}[theorem]{Question}

\begin{document}

\title{The Fiber of $Sq^n$}

\author{Robert R. Bruner}
\email{robert.bruner@wayne.edu}
\address{Mathematics Department,
         Wayne State University,
         Detroit, Michigan, 48067
         USA}


\begin{abstract}
A colleague asked about the Adams filtrations of the homotopy 
classes in the homotopy of the fiber of a particular map between GEMs.
Theorem 1.1 in [Bruner-Rognes, Trans AMS 2022]
proves to be effective in answering this (see Theorem 4.4).
We show that this and some related Adams spectral sequences all
collapse at $E_3$ and we determine the value of $E_3 = E_\infty$.
Notably, we do not need to determine the cohomology
of the fiber or the $E_2$ term of the Adams spectral sequence
to do this.
\end{abstract}


\received{Month Day, Year}   
\revised{Month Day, Year}    
\published{Month Day, Year}  
\submitted{Bill Murray}      
\volumeyear{2023} 
\volumenumber{25} 
\issuenumber{1}   
\startpage{401}     
\articlenumber{1} 


\maketitle


\section{Introduction}

Write $H$ for the mod 2 Eilenberg MacLane spectrum, $HZ$ for the integral
Eilenberg MacLane spectrum, and
$\cA$ for the mod 2 Steenrod algebra.  We will write
$\Ext(M) $ for $\Ext_\cA(M,\F_2) $.

For certain purposes it is useful to know not only the homotopy
groups of a spectrum, but also the Adams filtrations of the classes
involved.  A colleague asked me, in connection with applications in
surgery theory, how to determine these for the 
fiber $F$  of the map
\[ HZ \lra \prod_{i>0} \Sigma^{2i} H
\]
which is $Sq^{2i}$ to the $i^{\text{th}}$ factor.  It is easily seen that the
homotopy groups are $Z$ in degree $0$, and $Z/2$ in each positive odd degree.
However, the cohomology of the fiber is somewhat less evident, and it is
not hard to verify that $\Ext(H^*F)$ is sufficiently large that the Adams
filtrations of these homotopy classes is not immediately apparent.  Clearly,
differentials must intervene to result in an $E_\infty$ term which gives
the evident homotopy groups.  

This is exactly the situation faced in studying the Adams spectral
sequence for the connective image of J spectrum:  the homotopy
groups are evident from the fiber sequence, but the $E_2$ term of
the Adams spectral sequence is sufficiently large that the pattern
of differentials which produces these homotopy groups is not evident.
In~\cite{Davis}, Don Davis described the $E_2$ term and noted that a large
number of non-zero differentials were required to produce an $E_\infty$ term
consistent with the known homotopy groups.  In~\cite{ImJ},
John Rognes and the author proved a theorem which
addresses exactly this situation, showing that the $d_2$ differential
in  the Adams spectral sequence $\Ext(H^*\fib(f)) \Longrightarrow \pi_*\fib(f)$
is closely related to a 2-extension containing $f^*$.
We start by stating this result in Section~\ref{d2}.

Then, in the next two sections we study two variants of one factor of
the map in question, followed by a section in which we answer the
main question.   Notably, we do not have to calculate the cohomology
of the fiber or its Ext groups in any of these cases.   (We are not claiming that this is difficult, just that it is unnecessary.)   Theorem~\ref{d2thm} allows us
to determine the $E_3 = E_\infty$ term of the Adams spectral sequence directly, without
needing to have an explicit description of the $E_2$ term.  

This reminds the
author of the fact that the secondary Adams spectral sequence of Baues, Jibladze,  
Nassau, Frankland et al
(\cite{Baues}, \cite{BJ}, \cite{Nassau}, \cite{BF})
starts from the $E_3$ term of the classical Adams spectral sequence, and we
ask, in the final section, whether secondary cohomology has any bearing on
the results here.

\section{The theorem on $d_2$}
\label{d2}

Write $H$ for the mod 2 Eilenberg MacLane spectrum and
$\cA$ for the mod 2 Steenrod algebra.  We will write
$\Ext(M) $ for $\Ext_\cA(M,\F_2) $.

Let $X=Ff \lra Y \stackrel{f}{\lra} Z$ be a fiber sequence of spectra.   Factor its
long exact sequence in cohomology into short exact sequences:
\[
\xymatrix{
& H^*X
\ar@{->>}^{p_K}[ld]
&& H^*Y
\ar[ll]
\ar@{->>}^{p_C}[ld]
&& H^*Z
\ar_{f^*}[ll]
\ar@{->>}^{p_I}[ld]
&& H^* \Sigma X
\ar[ll]
\ar@{->>}^{p_K}[ld]
\\
\Sigma^{-1} K
&& C
\ar@{>->}^{i_C}[lu]
&& I
\ar@{>->}^{i_I}[lu]
&& K
\ar@{>->}^{i_K}[lu]
\\
}
\]
writing $K$, $I$ and $C$  for the kernel, image and cokenel of $f^*$.
These three short exact sequences of $\cA$-modules induce long exact sequences in
$\Ext$ with boundary maps $\partial_{KC} = \partial_X$,
$\partial_{CI} = \partial_Y$, and
$\partial_{IK} = \partial_Z$.

\begin{theorem}[\cite{ImJ}*{Theorem 1.1}]
\label{d2thm}
Let $d_2$ be the differential in the $E_2$ term of the 
Adams spectral sequence $\Ext(H^*X) \Longrightarrow \pi_*X$.
For each $(s,t)$ the composite
\[
\Ext^{s,t}(\Sigma^{-1} K) \xrightarrow{p_K^*}
\Ext^{s,t}(H^*X) \xrightarrow{d_2}
\Ext^{s+2,t+1}(H^*X) \xrightarrow{i_C^*}
\Ext^{s+2,t+1}(C) 
\]
is given by Yoneda composition with
\[
0 \lla C \lla H^*Y \stackrel{f^*}{\lla} H^*Z  \lla K \lla 0;
\]
equivalently, by the composite of the boundary maps
\[
\Ext^{s,t}(\Sigma^{-1}K) =
\Ext^{s,t+1}(K) \xrightarrow{\partial_{IK}}
\Ext^{s+1,t+1}(I) \xrightarrow{\partial_{CI}}
\Ext^{s+2,t+1}(C).
\]
\end{theorem}

\section{The fiber of $Sq^n : H \lra H$}

We first apply Theorem~\ref{d2thm} to the fiber sequence
$F_n \lra H \xrightarrow{Sq^n} \Sigma^n H$.

\begin{theorem}
\label{prop1}
The Adams spectral sequence 
\[
\Ext_{\cA}(H^*F_n,\F_2) \Longrightarrow \pi_* F_n
\]
collapses at $E_3$, with
\[
E_3^{*,*} = E_\infty^{*,*} = \F_2 \oplus \Sigma^{1,n}\F_2.
\]
\end{theorem}

Of course this is what we might expect for $E_\infty$,
but $E_2$ is very far from this.   
Theorem~\ref{d2thm} allows us to prove the result without
calculating  either $H^*F_n$ or the $E_2$ term explicitly.

As above, let $K$, $C$, and $I$ be the kernel, cokernel and image
of $\Sigma^n \cA \xrightarrow{Sq^n} \cA$, and
factor the long exact cohomology sequence
of $F_n \lra H \xrightarrow{Sq^n} \Sigma^n H$ into short exact 
sequences:
\[
\xymatrix{
& H^*F_n
\ar@{->>}^{p_K}[ld]
&& \cA
\ar[ll]
\ar@{->>}^{p_C}[ld]
&& \Sigma^n \cA
\ar_{Sq^n}[ll]
\ar@{->>}^{p_I}[ld]
&& H^* \Sigma F_n
\ar[ll]
\ar@{->>}^{p_K}[ld]
\\
\Sigma^{-1} K
&& C
\ar@{>->}^{i_C}[lu]
&& I
\ar@{>->}^{i_I}[lu]
&& K
\ar@{>->}^{i_K}[lu]
\\
}
\]
Theorem~\ref{d2thm} states that the composite
\[
\Ext^{s,t}(\Sigma^{-1} K) \xrightarrow{p_K^*}
\Ext^{s,t}(H^*F_n) \xrightarrow{\,d_2\,}
\Ext^{s+2,t+1}(H^*F_n) \xrightarrow{i_C^*}
\Ext^{s+2,t+1}(C) 
\]
is given by the composite of the boundary maps
\[
\Ext^{s,t}(K) \xrightarrow{\partial_{IK}}
\Ext^{s+1,t}(I) \xrightarrow{\partial_{CI}}
\Ext^{s+2,t}(C).
\]

\begin{lemma}
Both of these boundary maps are isomorphisms for $s\geq 0$ and
all $t$.

Further, $\Ext^0(C) = \F_2$ and $\Ext^1(C) = \Sigma^n\F_2$.
\end{lemma}

\begin{proof}
This is a standard consequence of the long exact sequences in $\Ext$
containing $\partial_{IK}$ and $\partial_{CI}$ since
$\Ext(\Sigma^n\cA) = \Sigma^{0,n}\F_2$.
\end{proof}

\begin{proof}[Proof of Proposition~\ref{prop1}]
By the Lemma, the composite $i_C^* d_2 p_K^* = \partial_{CI}\partial_{IK}$
is an isomorphism.
Hence $p_K^*$ is mono and $i_C^*$
is epi, so that $\Ext(H^*F_n)$ sits in a short exact sequence
\[
0 \lra \Ext^{s,t}(\Sigma^{-1}K) \xrightarrow{p_K^*} \Ext^{s,t}(H^*F_n) 
 \xrightarrow{i_C^*} \Ext^{s,t}(C)  \lra 0.
\]
The homology with respect to $d_2$ therefore consists of
classes $\Ext^{s,t}(H^*F_n)$ which map isomorphically to
$\Ext^{s,t}(C)$ for $s=0$ or $1$.  These are
$\Ext^{0,0}(C) = \F_2$ and $\Ext^{1,n}(C) = \F_2$.
\end{proof}

\section{The fiber of $Sq^n : HZ \lra H$}

If we let $F_nZ = \fib{(HZ \xrightarrow{Sq^n} \Sigma^n H)}$, the result
is nearly the same, as the proof is nearly the same:  the composite
\[
\Ext^{s,t}(K) \xrightarrow{\partial_{IK}}
\Ext^{s+1,t}(I) \xrightarrow{\partial_{CI}}
\Ext^{s+2,t}(C).
\]
remains a monomorphism, but now has cokernel $\F_2[h_0]$
rather than $\F_2$.

\begin{theorem}
The Adams spectral sequence 
\[
\Ext_{\cA}(H^*F_nZ,\F_2) \Longrightarrow \pi_* F_nZ
\]
collapses at $E_3$, so that
\[
E_3^{*,*} = E_\infty^{*,*} = \F_2[h_0] \oplus \Sigma^{1,n} \F_2
\]
with $h_0$ in bidegree $(s,t) = (1,1)$.
\end{theorem}


\section{The case of interest}

The case of interest combines the preceding maps for all even positive $n$.
Let $F$ be the fiber of
\[
HZ \lra \prod_{i > 0} \Sigma^{2i} H
\]
with components $Sq^{2i}$.   In cohomology this induces
\[
\xymatrix@C=4ex{
& H^*F
\ar@{->>}^{p_K}[ld]
&& \cA/\cA Sq^1
\ar[ll]
\ar@{->>}^{p_C}[ld]
&& \bigoplus \Sigma^{2i} \cA
\ar_{\bigoplus Sq^{2i}}[ll]
\ar@{->>}^{p_I}[ld]
&& H^* \Sigma F
\ar[ll]
\ar@{->>}^{p_K}[ld]
\\
\Sigma^{-1} K
&& C
\ar@{>->}^{i_C}[lu]
&& I
\ar@{>->}^{i_I}[lu]
&& K
\ar@{>->}^{i_K}[lu]
\\
}
\]
Clearly $C = \F_2$ and $I = \overline{\cA/\cA Sq^1}$, the kernel of
the nontrivial homomorphism ${\cA/\cA Sq^1} \lra \F_2$.
The boundary map $\partial_{CI}$ is particularly simple.

\begin{lemma}
In $\Ext$  the short exact sequence 
\[
0 \lla C=\F_2 \xleftarrow{~p_C} \cA/\cA Sq^1 \xleftarrow{i_I} I \lla 0
\]
induces the short exact sequence
\[
0 \lra  \Sigma^{-1,0}\Ext(I) \xrightarrow{\partial_{CI}}
\Ext(\F_2) \xrightarrow{p_C^*} \F_2[h_0] \lra 0
\]
so that  $\Ext^{s,t}(I) = \Ext^{s+1,t}(\F_2)$ for $t-s > 1$ and
$0$ otherwise.  This consists of
all of $\Ext(\F_2)$ except the $h_0$ tower in the 0-stem,
shifted down one in $s$ and hence up one in $t-s$.
\qed
\end{lemma}

Next we consider $\partial_{IK}$.

\begin{lemma}
The homomorphism $\partial_{IK}$ sits in an exact sequence
\begin{multline*}
0 \lra 
\left(\bigoplus_{\substack{i>0 \\ \,\,\,\, i \neq 2^j}} \Sigma^{0,2i}\F_2\right)^{s,t}
\lra
\Ext^{s,t}(K) \xrightarrow{\partial_{IK}}\\
\Ext^{s+1,t}(I) \lra
\left(\bigoplus_{j>0} \Sigma^{0,2^j} \F_2\right)^{s+1,t}
\lra 0.
\end{multline*}
The kernel of $\partial_{IK}$ consists of $\F_2$'s in 
$\Ext^{0,2i}(K)$ for all $2i$ that are not powers of $2$.
The image of $\partial_{IK}$ consists of the positive Adams
filtration elements in $\Ext(I)$.   This image is mapped isomorphically by
$\partial_{CI}$ to that part of $\Ext(C) = \Ext(\F_2)$
which has $t-s>0$ and is in Adams filtration $2$ or greater.
\end{lemma}

\begin{proof}
The short exact sequence
\[
0 \lla I \xleftarrow{\,p_I\,} \bigoplus_{i>0} \Sigma^{2i}\cA
\xleftarrow{i_K} K \lla 0
\]
induces
\[
\xymatrix@R=2ex@C=2ex{
\Sigma^{-1,0} \Ext(K)
\ar^(0.6){\partial_{IK}}[rr]
&&
\Ext(I)
\ar^{p_I^*}[rr]
\ar@{->>}[rd]
&&
{\ds{\bigoplus_{i>0} \Sigma^{2i}\F_2}}
\ar^{i_K^*}[rr]
\ar@{->>}[rd]
&&
\Ext(K).
\\
&&&
{\ds{\bigoplus_{j>0} \Sigma^{2^j} \F_2}}
\ar@{ >->}[ru]
&&
{\ds{\bigoplus_{\substack{i>0 \\ \,\,\,\, i \neq 2^j}} \Sigma^{2i}\F_2}}
\ar@{ >->}[ru]
}
\]
The summands of $\Ext(\bigoplus\Sigma^{2i} \cA) = \bigoplus \Sigma^{2i} \F_2$ split as they do
because, when $2i$ is a power of $2$,
$Sq^{2i}$ is indecomposable, hence not in the image of lower summands,
while, when $2i$ is not a power of $2$,
$Sq^{2i}$ is decomposable, hence already
in the image of the lower summands.
\end{proof}

Combining these two results, we have

\begin{lemma}
The composite 
\[
\partial_{CI}\partial_{IK} : \Ext(\Sigma^{-1} K) \lra \Ext(C) = \Ext(\F_2)
\]
has kernel 
\[
{\ds{\bigoplus_{\substack{i>0 \\ \,\,\,\, i \neq 2^j}} \Sigma^{2i}\F_2}}
\]
and cokernel $\F_2[h_0] \oplus \bigoplus_{j>0} \F_2\{h_j\}$.
\end{lemma}

Shifting this kernel
down one degree to account for the $\Sigma^{-1}$ in the domain
of $\partial_{CI}\partial_{IK}$, we reach our main result.

\begin{theorem}
\label{ofinterest}
The $E_3=E_\infty$ term of the Adams spectral sequence for $\pi_*F$
has $\F_2[h_0]$ in the $0$-stem, a single $\F_2$ in each positive odd stem,
and $0$ otherwise.
The $\F_2$ is
in Adams filtration $1$ in degrees $2^j-1$, and  in Adams filtration $0$
in degrees $2i-1$ when $i$ is not a power of $2$. \qed
\end{theorem}

Evidently, the isomorphism $\pi_{2i-1}F \lra \pi_{2i-1}F_{2i}Z$ induced by the projection map
\[
\xymatrix{
F
\ar[r]
\ar[d]
&
HZ
\ar@{=}[d]
\ar[r]
&
{\ds{\prod_{i>0} \Sigma^{2i} H}}
\ar^{{\text{proj}}_i}[d]
\\
F_{2i}Z
\ar[r]
&
HZ
\ar[r]
&
\Sigma^{2i} H
\\
}
\]
preserves Adams filtration when $i$ is a power of $2$,
while it raises Adams filtration
by $1$ when $i$ is not a power of $2$.

\begin{remark}
For geometric reasons it might be more natural to consider
the fiber $F'$ of $(\chi(Sq^{2i}))_{i}$.  Since $\chi(Sq^{2i})$
is decomposable iff $Sq^{2i}$ is, the argument above shows that
Theorem~\ref{ofinterest} applies equally well to $\pi_*F'$.
\end{remark}

\section{Secondary cohomology}

This may be an interesting test case for secondary cohomology,
since $\Ext$ in the secondary category gives the $E_3$-term directly.
Our calculations therefore tell us the secondary $\Ext$ modules for the 
secondary cohomology of the three fibers we considered.
This raises a specific question and a more general one.

\begin{question}
Can one compute the secondary cohomology of the fibers $F_n$, $F_nZ$ 
and $F$ and determine
the $E_3$-term of their Adams spectral sequences from this?
\end{question}

\begin{question}
Can Theorem~\ref{d2thm} be proved (or improved)
using secondary cohomology and the secondary Adams spectral sequence?
\end{question}

Positive answers to these questions do not seem to exist in the primary
sources,
\cite{Baues}, \cite{BJ}, \cite{Nassau}, \cite{BF}.

\end{document}